\documentclass[12pt]{article}

\usepackage{xcolor}
\usepackage{graphicx}
\usepackage{subcaption}
\usepackage{amsmath}
\usepackage{amssymb}
\usepackage{amsthm}

\theoremstyle{theorem}
\newtheorem{theorem}{Theorem}
\newtheorem{lemma}{Lemma}
\newtheorem{proposition}{Proposition}
\newtheorem{corollary}{Corollary}
 
\theoremstyle{definition}
\newtheorem*{definition}{Definition}
\newtheorem*{remark}{Remark}
\newtheorem*{problem}{Problem}
\newtheorem*{example}{Example}

\def\Q{\mathbb{Q}}
\def\O{\mathcal{O}}
\def\Z{\mathbb{Z}}
\def\R{\mathbb{R}}
\def\C{\mathbb{C}}
\def\T{T}

\def\Ball{B_{1/k}(\alpha_{i,j}^k)}
\DeclareMathOperator\im{Im}

\DeclareMathOperator\denom{denom}
\DeclareMathOperator\rk{rank}
\DeclareMathOperator\Cl{Cl}

\begin{document}

\title{Euclidean rings of $S$-integers in complex quadratic fields}
\author{Kyle Hammer; Kevin McGown; Skip Moses}
\markright{Euclidean rings of $S$-integers}

\maketitle

\begin{abstract}
We give an elementary approach to studying whether rings of $S$-integers in complex quadratic fields
are Euclidean with respect to the $S$-norm.
\end{abstract}


\section{Introduction}\label{S:intro}

In a first course on elementary number theory, one 
establishes the Fundamental Theorem of Arithmetic,
which says that every positive integer factors uniquely as product of primes.
In the standard proof, one needs to know that for any prime $p$,
if $p$ divides $ab$, then $p$ divides $a$ or $p$ divides $b$.
This seemingly obvious fact about primes is usually derived from the Euclidean algorithm.
The following crucial property of $\Z$ allows one to show that the Euclidean algorithm
terminates in a finite number of steps:  For all $a,b\in\Z$, $b\neq 0$, there exists $q,r\in\Z$
such that $a=qb+r$ and $|r|<|b|$.  Equivalently, this ``Euclidean property'' reads: 
For all $x\in\Q$ there exists $c\in\Z$ such that $|x-c|<1$.

Algebraic number fields are one of the main objects of study in algebraic number theory.
The simplest number fields other than the rational numbers~$\Q$ are the complex
quadratic fields which take the form
$$K=\Q(\sqrt{-d})=\{x+y\sqrt{-d}\mid x,y\in\Q\}\,,$$
where $d>0$ is squarefree.
We view $K$ as a subset of the complex numbers $\C$.
Let $N:K\to\Q$ denote the norm function (which multiplies an element by its conjugate)
$$N(x+y\sqrt{-d})=(x+y\sqrt{-d})(x-y\sqrt{-d})=x^2+dy^2.$$
One has $N(\xi\eta)=N(\xi)N(\eta)$ for all $\xi,\eta\in K$,
as is verified by direct computation.
Let $\O$ denote the ring of integers in $K$, which consists of those elements that satisfy a monic polynomial 
with integer coefficients.  One can show $\O=\Z[w]=\{a+bw\mid a,b\in\Z\}$ with
$$
  w=\begin{cases}
  \sqrt{-d} & -d\equiv 2,3\pmod{4}\\[1ex]
  \frac{1+\sqrt{-d}}{2} & -d\equiv 1\pmod{4}
  \,.
  \end{cases}
$$
The ring $\O$ in $K$ is the analogue of $\Z$ in $\Q$, but properties enjoyed by $\Z$
(such as unique factorization) do not always hold in $\O$.

We now give the classical generalization of the Euclidean algorithm.
We call $K$ norm-Euclidean if for every $\xi\in K$ there exists $\gamma\in\O$ such that
$N(\xi-\gamma)<1$; this is equivalent to the condition that $\O$ is a Euclidean ring with respect to
the function $N$.
It is known that $K$ is norm-Euclidean exactly when $d=-1,-2,-3,-7,-11$.
This goes back to Dedekind's supplement to Dirichlet's book~\cite{DD}.

Let $S\subseteq\Z^+$ be a finite (possibly empty) set of primes,
and let $T\subseteq\Z^+$ be the set of all finite products of elements in $S$
(where $1\in T$ by convention).
In this paper, we are interested in the so-called ring of $S$-integers
$$
  \O_{S}=\left\{\frac{a+bw}{c}: a,b,\in\Z, c\in T\right\}\subseteq K\,.
$$
This is an example of the ring-of-fractions construction one might encounter in a first course on ring theory.

We define the $S$-norm of an element $\xi\in K$, denoted by $N_S(\xi)$, by deleting the primes in $S$ from the prime factorization of
the numerator and denominator of the rational number $N(\xi)$.
For example, when $K=\Q(\sqrt{-5})$, $S=\{2\}$, $\xi=(3+3\sqrt{-5})/7$ one has
$N(\xi)=54/49$ and hence $N_S(\xi)=27/49$.
It follows from the multiplicative property of $N$ that the function $N_S$ is also multiplicative.

We call $K$ $S$-norm-Euclidean if 
for every $\xi\in K$ there exists $\gamma\in\O_S$ such that $N_S(\xi-\gamma)<1$;
this is equivalent to the statement that $\O_S$ is a Euclidean ring with respect to the function $N_S$.
One would like to know: When is $K$ $S$-norm-Euclidean?
Even in our setting, where $K$ is a complex quadratic field,
this question is, in general, unsolved.
To give a simple example, if $K=\Q(\sqrt{-5})$ and $S=\{2\}$,
then $\O_S=\Z\left[\sqrt{-5},\frac{1}{2}\right]$; in this case $K$ is $S$-norm-Euclidean,
although $K$ is not norm-Euclidean.

In this paper we give an elementary approach that allows us to
prove a few results about $S$-norm-Euclidean fields in the complex quadratic setting.
Our first result furnishes a number of examples.

\begin{theorem}\label{T:1}
The complex quadratic field $K=\Q(\sqrt{-d})$ is $S$-norm-Euclidean
for the following choices of $d$ and $S$:

\begin{tabular}{c|l}
$S$ & Values of $d$\\
\hline
$\emptyset$ & 1,2,3,7,11\\
\hline
$\{2\}$ & 1,2, 3, 5, 6, 7, 11, 15, 19, 23\\
\hline
$\{2,3\}$ & 1,2, 3, 5, 6, 7, 10, 11, 13, 14, 15, 17, 19, 23, 31, 35, 39, 43, 47, 51,\\& 55, 59, 67, 71\\
\hline
$\{2,3,5\}$ & 1,2, 3, 5, 6, 7, 10, 11, 13, 14, 15, 17, 19, 21, 22, 23, 26, 29, 30, 31,\\& 33, 34, 35, 39, 43, 47, 51, 55, 59, 67, 71, 79, 83, 87, 91, 95, 103,\\& 107, 111, 115, 119, 123, 127, 131, 139, 143
\end{tabular}
\end{theorem}
Note the previous theorem is not claiming these lists are complete or even finite,
although both are true when $S=\emptyset$.
As discussed, when $S=\emptyset$, this result is classical.
When $S=\{2\}$ these examples appear in~\cite{LENSTRA,VDL}.
However, we have not seen any examples with $S=\{2,3\}$ explicitly given in the literature.

A natural question arises:  Does there always exists a set $S$ so that $K$ is $S$-norm-Euclidean?
The answer to this question is yes, and we give a quantified version of this claim.
Before stating this result, we require an additional piece of terminology and notation.
We write $D$ to denote the absolute value of the discriminant of $K$.
By definition, the discriminant is $\det([1, w; 1, \overline{w}])^2$ where
$\overline{w}$ denotes the complex conjugate of $w$;
therefore the discriminant equals $-D$ where
$$
D=
\begin{cases}
  4d & -d\equiv 2,3\pmod{4}\\
  d & -d\equiv 1\pmod{4}.
\end{cases}
$$
Given $x\in\R$ we write $\lceil x\rceil$ to denote the ceiling of $x$;
i.e., $\lceil x\rceil$ is the smallest integer greater than $x$.

\begin{theorem}\label{T:2}
  Let $K$ be an complex quadratic field of discriminant $-D$.
  If $S$ contains all primes less than
  $\lceil\sqrt{D}/\sqrt{3}\rceil$,
  then $K$ is $S$-norm-Euclidean.
\end{theorem}
In a first course in algebraic number theory, one shows that the class group
is generated by the set of all prime ideals whose norm is less than the Minkowski bound;
it turns out that Theorem~\ref{T:2} implies this result in our setting.  However our proof
of Theorem~\ref{T:2} does not use the Minkowski bound, the class group, or even the notion of an ideal!
For those familiar with the rudiments of algebraic number theory,
further discussion is given in an appendix (see~\S\ref{S:appendix}).
We should note that Theorem~\ref{T:2} appears in~\cite{MARKANDA} but that our proof is different.

To give another example, consider
$K=\Q(\sqrt{-163})$ and $S=\{2,3,5,7\}$, where
$$
\O_S=\Z\left[\frac{1+\sqrt{-163}}{2}, \frac{1}{120}\right]
=
\Z\left[\sqrt{-163}, \frac{1}{120}\right]
\,.
$$
Theorem~\ref{T:2} implies that $K$ is $S$-norm-Euclidean;
however, $K$ is not norm-Euclidean (even though $\O=\Z\left[(1+\sqrt{-163})/2\right]$ is a PID and hence a UFD).

On the other hand, one might consider a fixed set $S$ and ask
which complex quadratic fields $K$ are $S$-norm-Euclidean.
When $S$ contains several elements, this seems to be a difficult problem,
so we restrict ourselves to the case where $S$ contains a single prime $p\in\Z$.
Even in this simplest case, we will further restrict which $(d,p)$ pairs we consider
in order to obtain a complete classification.
In what follows $(n/p)$ denotes the usual Legendre symbol;
namely $(n/p)=1$ if $x^2\equiv n\pmod {p}$ has a solution with $x\not\equiv 0\pmod{p}$,
$(n/p)=0$ if $p$ divides $n$, and $(n/p)=-1$ otherwise.
%
%
\begin{theorem}\label{T:3}
Let $K=\Q(\sqrt{-d})$ with $d>0$ squarefree and $S=\{p\}$.
\begin{enumerate}
\item
Suppose $p\geq 11$ and $(-d/p)\neq 1$.
Then $K$ is $S$-norm-Euclidean only if 
$d\in\{1,2,3,7,11\}$. 
\item
Suppose $p=2$ and $-d\not\equiv 1\pmod{8}$.
Then $K$ is $S$-norm-Euclidean if and only if
$d\in\{1,2,3,5,6,10,11,19\}$.
\item
Suppose $p\in\{3,5,7\}$ and $(-d/p)\neq 1$.  Then $K$ is $S$-norm-Euclidean if and only if
one of the following holds:
\begin{align*}
&p=3,\; d\in\{1,3,7,11,15\}\\
&p=5,\; d\in\{2,3,7,15,35\}\\
&p=7,\; d\in\{1,2,7,11,35\}
\end{align*}
\end{enumerate}
\end{theorem}
We note that the previous theorem follows from Theorem~0.19 of~\cite{VDL};
however, our approach requires very little background knowledge and
our proof is relatively short.
In comparing Theorem~\ref{T:1} with Theorem~\ref{T:3}, 
the reader may have noticed the addition of $d=10$ to the list when $S=\{2\}$.
It turns out the technique used to prove Theorem~\ref{T:1} does not quite apply
in the case of $(d,p)=(10,2)$,
but we find a way to treat this case and four other ``exceptional cases''
in Section~\ref{S:except}.

We view the following question as a natural one.
In light of Theorem~\ref{T:3} one can restrict attention to those $(d,p)$ pairs such that $(-d/p)=1$.

\begin{problem}
Are there infinitely many $d$ for which there exists $p$ such that $K$ is $S$-norm-Euclidean when $S=\{p\}$?
\end{problem}
Associated to any number field $K$ is a finite abelian group, called the class group of $K$.
Roughly speaking, this group governs the complexity of factorizations of elements of $\O$
into irreducible elements, and this group is trivial if and only if unique factorization holds.
It turns out that if $K$ is $S$-norm-Euclidean with $S=\{p\}$, then the class group of $K$ is cyclic.
According to the Cohen--Lenstra heuristics (see~\cite{CL}),
the latter event should happen quite often (roughly 97.75\% of the time) for fields of prime discriminant,
but we currently cannot even prove that it happens infinitely often.
(See \S\ref{S:appendix} for additional discussion.)
If one were to resolve the problem above in the affirmative,
then one would prove that class group of $K$ is cyclic infinitely often,
resolving a major open problem.
(Of course, this likely means that producing such a proof may be difficult.)
On the other hand, if the resolution of this problem is in the negative,
then proving such a result would be interesting in its own right.
We note that Stark has already suggested looking at Euclidean rings of $S$-integers
in connection with class number problems (see~\cite{STARK}).

Finally, since we are posing problems anyway, here is another:

\begin{problem}
Are there infinitely many $d$ for which $K$ is $S$-norm-Euclidean when $S=\{2,3\}$?
\end{problem}




%
%
%
%


\section{Main idea}\label{S:main}

In this section, we develop the main idea of this paper,
which provides sufficient conditions for $K$ to be $S$-norm-Euclidean.
We will see that Theorems~\ref{T:1} and~\ref{T:2} will then immediately follow.
We will give a small amount of additional background material along the way as necessary.
Let $K$ denote a complex quadratic field and adopt all the notation given in~\S\ref{S:intro}.
In particular, we write  $K=\Q(\sqrt{-d})$ with $d>0$ squarefree and discriminant $-D$.

The norm function $N:K\to\Q$ also maps $\O\to\Z$.  In the case $-d\equiv 2,3\pmod{4}$ this is evident from our definition;
in the case where $-d\equiv 1\pmod{4}$, observe that
$$
  N(a+bw)=\left(a+\frac{b}{2}\right)^2+d\left(\frac{b}{2}\right)^2=a^2+ab+\left(\frac{d+1}{4}\right)b^2
  \,,
$$
where $(d+1)/4\in\Z$.
It follows from this that $N_S(\alpha)\leq N(\alpha)$ for all $\alpha\in\O$.
The following lemma gives an inequality that involves elements of $K$
that are not integral.
\begin{lemma}\label{lem:2}
Suppose
$$
  \xi=\frac{x+yw}{z}\in K
  \,,\quad
  \alpha = \frac{a+bw}{c} \in \mathcal{O}_{S}
$$
with $x,y,z\in\Z$, $z$ not divisible by any prime in $S$,
$a,b \in \mathbb{Z}$, $c\in\T$. Then 
$$N_S(\xi - \alpha) \leq c^2N(\xi - \alpha)\,. $$
\end{lemma}
\begin{proof}
Observe 
$$N(\xi - \alpha) = \frac{N(cz\xi-cz\alpha)}{z^2 c ^2}$$
and $cz\xi-cz\alpha\in\O$
which implies 
$$N_S(\xi - \alpha) \leq \frac{N(cz\xi-cz\alpha)}{z^2} = c^2N(\xi - \alpha)\,.$$
\end{proof}

We define the fundamental domain for $K$ as
$$
  \mathcal{F}=\{x+yw\mid 0\leq x,y\leq 1\}\subseteq\C
  \,.
$$
This will assist us in visualizing the norm-Euclidean property graphically.
When $-d\equiv 2,3\pmod{4}$, $\mathcal{F}$ takes the shape of a rectangle, and when $-d\equiv 1\pmod{4}$, $\mathcal{F}$
takes the shape of a parallelogram.  Notice that any $z\in\C$ can be written as $z=z'+\alpha$ where $z'\in\mathcal{F}$ and $\alpha\in\O$.
Indeed, translations of $\mathcal{F}$ give a tessellation of $\C$.
Hence, when studying the norm-Euclidean property, we can focus our attention on the points in $\mathcal{F}$.
Indeed, observe that $K$ is norm-Euclidean if and only if for all $\xi\in K\cap\mathcal{F}$
there exists $\gamma\in\O$ such that $|\xi-\gamma|<1$.
In particular, $N(\xi-\gamma)=|\xi-\gamma|^2$ where $|\cdot|$ denotes the complex modulus.

\begin{definition}
For $\alpha\in\O_{S}$ let $\denom_S(\alpha)$ denote
the minimal $c$ over all representations $\alpha=(a+bw)/c$ with $a,b\in\Z$, $c\in\T$.
\end{definition}

\begin{lemma}\label{lem:3}
We have $K$ is $S$-norm-Euclidean if
for all $\xi\in K\cap\mathcal{F}$ there exists $\alpha\in\O_{S}$
such that $|\xi-\alpha|<1/\denom_S(\alpha)$.
\end{lemma}

\begin{proof}
Let $\xi\in K$ be arbitrary.  We must show that there exists
$\alpha\in\O_{S}$ such that $N_S(\xi-\alpha)<1$.
Without loss of generality, we may assume that 
$\xi=(x+yw)/z$ where $z$ is not divisible by any primes in $S$,
by multiplying $\xi$ by an element of $T$ if necessary;
indeed, if $u\in T$, then $N_S(u\xi-\alpha)=N_S(u)N_S(\xi-u^{-1}\alpha)$
and $N_S(u)=1$.
Moreover, we may assume $\xi\in\mathcal{F}$ by subtracting an appropriate element of $\O$.
Noting that $N(\xi - \alpha) = |\xi - \alpha|^2$, Lemma~\ref{lem:2}
gives the result.
\end{proof}

\begin{example}
Consider the case of $K = \mathbb{Q}(\sqrt{-5})$ whose fundamental domain is depicted in Figure~\ref{Fig1}.
As one can see, we cannot cover $\mathcal{F}$ by circles of radius $1$ centered at the points of $\mathcal{O}$.
However, in light of Lemma~\ref{lem:3}, if we choose $S=\{2\}$, 
we can additionally consider circles of radius $1/2$ centered at points in $\mathcal{O}_{S}$
of the form $(a+bw)/2$.  In this case, we can cover $\mathcal{F}$ as depicted in Figure~\ref{Fig1}.
This allows us to conclude that $K$ is $S$-norm-Euclidean with $S=\{2\}$.

\begin{figure}[h!]
  \centering
  \begin{subfigure}[b]{0.3\linewidth}
    \includegraphics[width=\linewidth]{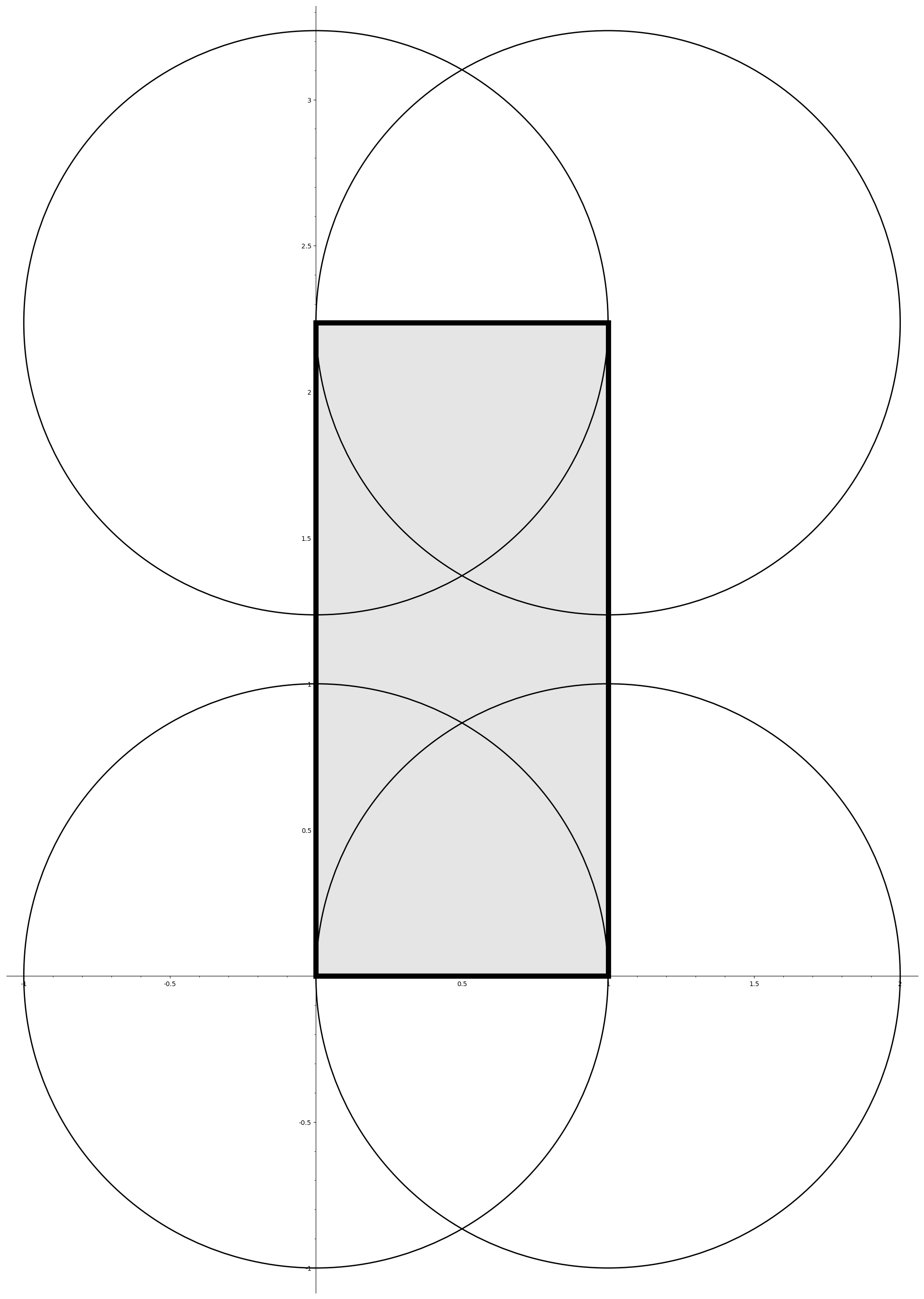}
    \caption{$S=\emptyset$}
  \end{subfigure}
  \begin{subfigure}[b]{0.3\linewidth}
    \includegraphics[width=\linewidth]{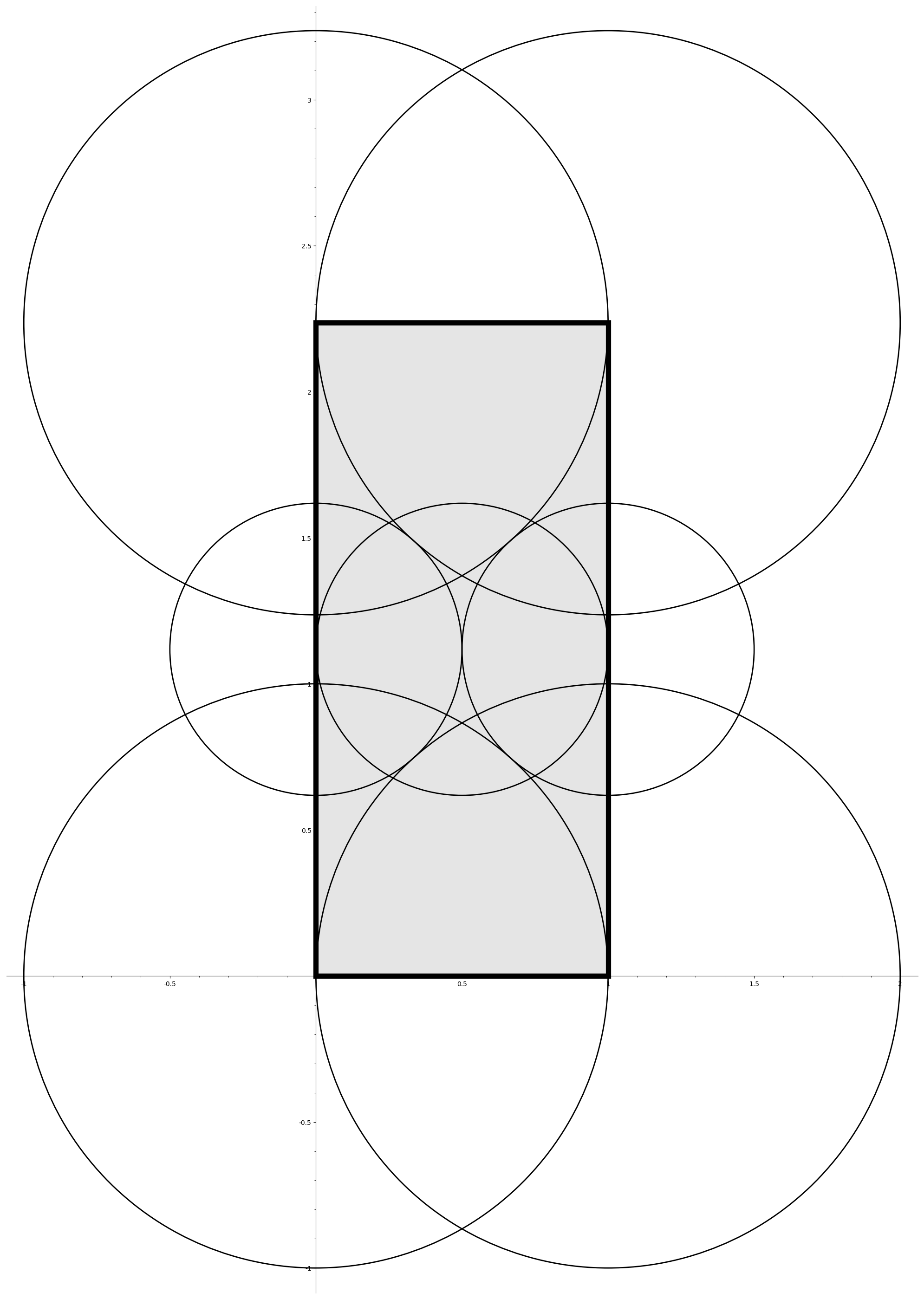}
    \caption{S=\{2\}}
  \end{subfigure}
  \caption{$K=\Q(\sqrt{-5})$}
  \label{Fig1}
\end{figure}
\end{example}

We now generalize the idea presented in the previous example.
Consider the collection of all $S$-integers in the fundamental domain (i.e., elements of $\O_{S}\cap \mathcal{F}$).
This collection is precisely the set of all $\alpha_{i,j}^k = (i+j\omega)/k$
for $k\in\T$ and $0\leq i,j\leq k$.
In light of Lemmas~\ref{lem:2} and~\ref{lem:3}, we know that the condition $\mathcal{F}\subseteq \bigcup_{i,j,k} \Ball$
is sufficient for $K$ to be $S$-Euclidean.  Moreover, we will only need to consider the $B_{1/k}(\alpha_{i,j}^k)$ where $\gcd(j,k)=1$.
To to each ``row of circles'' $\bigcup_{i=0}^k \Ball$, we can associate a
horizontal strip
\[
  R_j^k=\left\{x+iy\in\mathcal{F}:
  \left|y-\frac{j\im(w)}{k}\right|<\frac{\sqrt{3}}{2k}
  \right\}
  \,.
\]
Note that the $R_{j}^k$ are maximal with the property
that $R_{j}^k\subseteq \mathcal{F}\cap\bigcup_i \Ball$.
The $R_{j}^k$ give rise to intervals via the projection map $\pi(x+iy)= y/\im w$.
We define
$$
I_{j}^k = \left(\frac{j}{k} - \frac{\sqrt{3}}{2k \im \omega},\; \frac{j}{k} + \frac{\sqrt{3}}{2k \im \omega}\right)
=
\left(\frac{j- \sqrt{3/D}}{k},\; \frac{j + \sqrt{3/D}}{k}\right)
\,.
$$
Notice that $\im(w)=\sqrt{D}/2$, which allowed us to rewrite
the intervals in terms of $D$ rather than $\im w$.
The following lemma gives sufficient conditions for $K$ to be $S$-norm-Euclidean.
\begin{lemma}\label{L:balls}
Notation as above.
If $[0,1]\cap\Q\subseteq \bigcup_{j,k} I_{j}^k$, then $K$ is $S$-norm-Euclidean.
\end{lemma}
\begin{proof}
Suppose $[0,1]\cap\Q\subseteq \bigcup_{j,k} I_{j}^k$.
Taking $\pi^{-1}$ of both sides allows use to observe that
\[
 K\cap\mathcal{F}\subseteq \bigcup_{j,k}\pi^{-1}(I_{j}^k)
 \subseteq
 \bigcup_{i,j,k} B_{1/k}(\alpha_{i,j}^k)
 \,.
\]
In light of the preceding discussion, we conclude that $K$ is $S$-norm-Euclidean.
\end{proof}

\begin{example}
Consider the field $K=\Q(\sqrt{-67})$
and choose $\{2,3\}$.
We write down the intervals $I_j^k$ with $0\leq j\leq k\leq 4$ and $\gcd(j,k)=1$,
and check that these cover the unit interval.  Thus $K$ is $S$-norm-Euclidean.
Figure~\ref{fig2} gives the intervals $I_j^k$ and the corresponding covering of $\mathcal{F}$.
We have sorted the intervals to make it easier to ``see'' them in the covering.


\begin{figure}[h!]
  \centering
  \begin{subfigure}[b]{0.4\linewidth}
    \includegraphics[width=\linewidth]{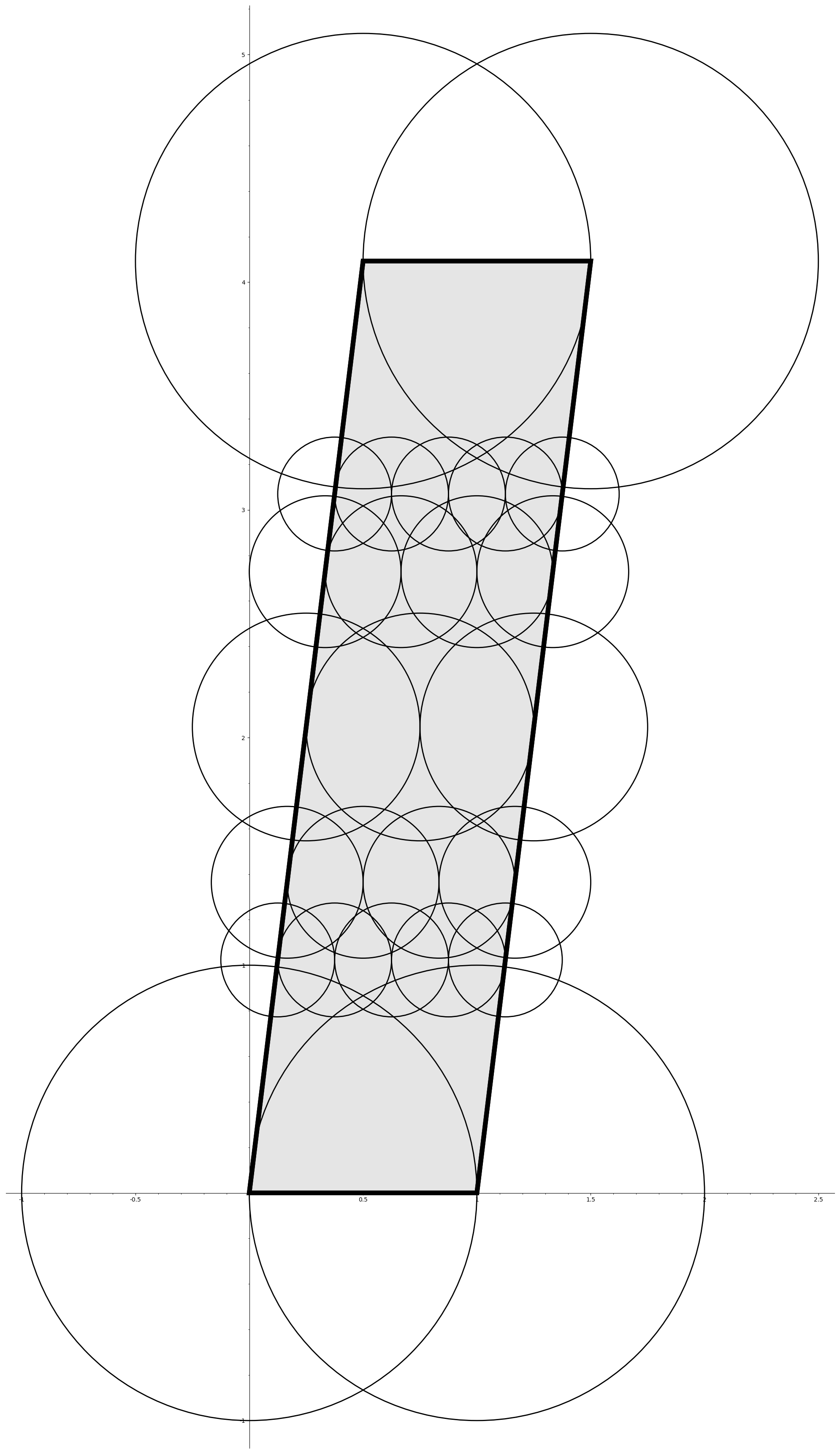}
    \caption{Covering $\mathcal{F}$ by $\bigcup\Ball$}
  \end{subfigure}
  \begin{subfigure}[b]{0.4\linewidth}
  \centering
  \begin{tabular}{c}
  $\left(\frac{1-\sqrt{3/67}}{1},\;\frac{1+\sqrt{3/67}}{1}\right)$
  $\approx \left(0.788, 1.211\right)$
  \\[2.5ex]
    $\left(\frac{3-\sqrt{3/67}}{4},\;\frac{3+\sqrt{3/67}}{4}\right)$  
    $\approx\left(0.697, 0.802\right)$
  \\[2.5ex]
  $\left(\frac{2-\sqrt{3/67}}{3},\;\frac{2+\sqrt{3/67}}{3}\right)$
      $\approx\left(0.596, 0.737\right)$
  \\[2.5ex]
    $\left(\frac{1-\sqrt{3/67}}{2},\;\frac{1+\sqrt{3/67}}{2}\right)$
        $\approx\left(0.394, 0.605\right)$
  \\[2.5ex]
  $\left(\frac{1-\sqrt{3/67}}{3},\;\frac{1+\sqrt{3/67}}{3}\right)$
      $\approx\left(0.262, 0.403\right)$
  \\[2.5ex]
  $\left(\frac{1-\sqrt{3/67}}{4},\;\frac{1+\sqrt{3/67}}{4}\right)$
      $\approx\left(0.197, 0.302\right)$
  \\[2.5ex]  
  $\left(\frac{-\sqrt{3/67}}{1},\;\frac{\sqrt{3/67}}{1}\right)$
      $\approx\left(-0.21, 0.211\right)$
  \end{tabular}
      \vspace{6ex}
    \caption{The intervals $I_j^k$}
  \end{subfigure}
  \caption{$K=\Q(\sqrt{-67})$}
      \label{fig2}
\end{figure}
\end{example}


We rephrase Lemma~\ref{L:balls} in a slightly different form, that will be convenient in the sequel.
In what follows, we write $\{x\}=x-\lfloor x\rfloor$ to denote the fractional part of a real number $x$.
\begin{corollary}\label{C:main}
Suppose that for every rational $y\in[0,1]$ there exists $c\in\T$ such that
$$\{c y\}<\frac{\sqrt{3}}{\sqrt{D}}
\;\text{  or  }\;
\{c y\}>1-\frac{\sqrt{3}}{\sqrt{D}}
\,.
$$
Then $K$ is $S$-norm-Euclidean.
\end{corollary}

\begin{remark}
Let $q$ be the smallest prime such that $q\not\in S$.
Since $\{c/q\}\geq 1/q$ for all $c\in T$,
the condition in Corollary~\ref{C:main} can only be satisfied when $D\leq 3q^2$.
\end{remark}

As illustrated in the previous example, our results thus far lead to the following simple procedure:
\begin{enumerate}
\item
Let $q$ be the smallest prime not in $S$ and set $X=3q^2$.
\item
Enumerate the intervals $I_{j}^k$ for $0\leq j\leq k\leq X$ with $\gcd(j,k)=1$.
\item
If $[0,1]\subseteq \bigcup_{j,k} I_{j}^k$, then $K$ is $S$-norm-Euclidean.
\end{enumerate}
Notice that this procedure will never prove that $K$ is not $S$-norm-Euclidean.
However, as was stated in Theorem~\ref{T:1} it allows us to show a number of complex quadratic fields $K$ are, in fact, $S$-norm-Euclidean.

\section{Proof of Theorems \ref{T:1} and \ref{T:2}}\label{S:proofs}

\begin{proof}[Proof of Theorem~\ref{T:1}]
The procedure described in the previous section leads to the proof of Theorem~\ref{T:1}.
Implementing this procedure on a computer gives the result immediately, but
this could be done by hand as the number of intervals necessary is not large.
Below we give the value of $X$ necessary and the number of intervals produced.

\begin{center}
\begin{tabular}{c |c| c}
$S$ & $X$ & $\#\{I_j^k\}$\\
\hline
$\emptyset$ & $1$ & $2$\\
$\{2\}$ & $2$ & $3$\\
$\{2,3\}$ & $4$ & $7$\\
$\{2,3,5\}$ & $6$ & $13$\\
\end{tabular}
\end{center}

\end{proof}

\begin{proof}[Proof of Theorem~\ref{T:2}]
The proof is a simple application of Corollary~\ref{C:main}.
Let $y\in[0,1]\cap \Q$.  Set $X=\sqrt{D}/\sqrt{3}$.
Suppose that $S$ contains all primes $\leq \lceil X\rceil $
so that $T$ contains all positive integers $\leq \lceil X \rceil$.
We will follow the proof of Dirichlet's approximation theorem
to show that there exists a $c \in \T$ such that $\{cy\} < 1/X$
or $\{cy\}>1-1/X$.
Subdivide the unit interval as
$$
[0,1) = \left[0,\frac{1}{X}\right)
\cup
\left[\frac{1}{X},\frac{2}{X}\right)
\cup\ldots \cup
\left[\frac{k-1}{X},1\right)
$$
where $k=\lceil X\rceil$ is the number of sub-intervals.
Each of the $k+1$ numbers
$\{0y\}, \{y\}, \ldots \{ky\}$
must land in one of the $k$ subintervals.
By the Pigeonhole Principle, there must exist an $s,t \in \mathbb{Z}$ satisfying $0 \leq s < t \leq \lceil X \rceil $ such that $\{sy\},\{ty\}$ differ by less than $X$.
Upon setting $c=t-s$ we obtain the desired result.
\end{proof}

\section{Proof of Theorem~\ref{T:3}}

Since Theorem~\ref{T:3} gives a classification of sorts, the proof will require some effort.
The first part concerns the case where $S=\{p\}$ with $p$ odd,
and the second part concerns the case where $S=\{2\}$.
In each part we must treat the cases of
$-d\equiv 2,3\pmod{4}$ and $-d\equiv 1\pmod{4}$ separately.
As first pass, the reader may wish to skip one or more of these subproofs.

\begin{proof}[Proof of Theorem~\ref{T:3} ($p$ is odd)]
Let $d>0$ be squarefree and $S=\{p\}$ with $p$ odd and $(-d/p)\neq 1$.  Set $\xi_0=(1+w)/2$.
Let $\alpha\in\O_{S}$ be arbitrary and write
$\alpha=(a+bw)/p^n$ where $a,b,n\in\Z$, $n\geq 0$, with $a$ and $b$ not both divisible by $p$.
Then we have
$$
  \alpha-\xi_0=\frac{(2a-p^n)+(2b-p^n)w}{2\cdot p^n}=\frac{p^m(A+Bw)}{2\cdot p^n}
  \,,
$$
where we have rewritten $2a-p^n=p^m A$ and $2b-p^n=p^mB$ with $m\geq 0$ maximal.
Note that $A,B\neq 0$ and both of $A$,$B$ cannot be divisible by $p$.

First consider the case where $-d\equiv 2,3\pmod{4}$.
We have
\begin{equation}\label{eqn:normp}
  N(\alpha-\xi_0)=\frac{p^{2m}(A^2+dB^2)}{2^2\cdot p^{2n}}
  \,.
\end{equation}
Suppose $(-d/p)=-1$.  We claim that $A^2+dB^2$ is not divisible by $p$.
Indeed if, $p$ divides $A^2+dB^2$, then we must have $A,B\not\equiv 0\pmod {p}$, lest
both $A$ and $B$ are divisible by $p$; in this case, one can solve $A^2+dB^2\equiv 0\pmod{p}$
to show that $-d$ is a square mod $p$, a contradiction.
Therefore $N_S(\alpha-\xi_0)\geq (1+d)/4\geq 1$ when $d\geq 3$.
We turn to the case where $(-d/p)=0$.  This is the same except that 
$A^2+dB^2$ is possibly divisible by one copy of $p$.
If $A+dB^2$ is not divisible by $p$, the result follows as before, so we may assume $p$ divides $A^2+dB^2$.
In this case, $A$ must be divisible by $p$ as well,
which implies that $A^2+dB^2\geq p^2+d$ and hence $N_S(\alpha-\xi_0)\geq (p^2+d)/(4p)\geq 1$
when $d\geq 3$.  Consequently, $K$ is not $S$-norm-Euclidean when $-d\equiv 2,3\pmod{4}$ and $d\geq 3$.

Next we turn to the case $-d\equiv 1\pmod{4}$.
We have
$$
 N(\alpha-\xi_0)=\frac{p^{2m}(A^2+AB+\frac{1+d}{4}B^2)}{2^2\cdot p^{2n}}
 =
 \frac{p^{2m}((2A+B)^2+dB^2)}{16\cdot p^{2n}}
 \,.
$$
If $(-d/p)=-1$, then $(2A+B)^2+dB^2$ is not divisible by $p$, and it follows that
then $N_S(\alpha-\xi_0)\geq (1+d)/16\geq 1$ when $d\geq 15$.
If $(-d/p)=0$ then $(2A+B)^2+dB^2$ is divisible by at most one copy of $p$.
As before, we may assume $p$ divides $(2A+B)^2+dB^2$.  In this case,
we find that $2A+B$ is divisible by $p$ and thus $N_S(\alpha-\xi)\geq (p^2+d)/(16p)$.
This gives $N_S(\alpha-\xi_0)\geq 1$ when $d\geq 15$ except when $(d,p)$ equals one of the pairs
$(15,3), (15,5), (35,5), (35,7)$.
We deal with the remaining cases in Section~\ref{S:except}.
\end{proof}

\begin{proof}[Proof of Theorem~\ref{T:3} ($p=2$)]
Let $d>0$ be squarefree and $S=\{2\}$.  Set $\xi_0=(1+w)/3$.
Let $\alpha\in\O_{S}$ be arbitrary and write
$\alpha=(a+bw)/2^n$ where $a,b,n\in\Z$, $n\geq 0$.
We aim to show that $N_S(\alpha-\xi_0)\geq 1$.
Observe that
$$
  \alpha-\xi_0=\frac{(3a-2^n)+(3b-2^n)w}{3\cdot 2^n}
  \,.
$$
We rewrite $3a-2^n=2^m A$ and $3b-2^n=2^m B$ where $A,B\in\Z$ and $m\geq 0$ is maximal;
note that $A,B\neq 0$ and that $A$ and $B$ cannot both be even.

First consider the case where $-d\equiv 2,3\pmod{4}$.
We have
\begin{equation}\label{eqn:1}
  N(\alpha-\xi_0)=\frac{2^{2m}(A^2+dB^2)}{3^2\cdot 2^{2n}}
  \,.
\end{equation}
We claim that $A^2+dB^2$ is divisible by at most one power of $2$;
indeed, the assumption $A^2+dB^2\equiv 0\pmod{4}$ together with $-d\equiv 2,3\pmod{4}$
leads to $A,B\equiv 0\pmod{2}$, a contradiction.
Whence
\begin{equation}
  N_S(\alpha-\xi_0)\geq\frac{A^2+dB^2}{2\cdot 3^2}\geq \frac{1+d}{18}
  \,,
\end{equation}
which implies $N_S(\alpha-\xi_0)\geq 1$ when $d\geq 17$.
Since $\alpha\in\O_S$ was arbitrary, this proves that $K$ is not $S$-norm-Euclidean
when $-d\equiv 2,3\pmod{4}$ and $d\geq 17$.

This leaves open the cases of $d=10,13,14$.
However, if $d$ is even then the assumption that $2$~divides $A^2+dB^2$ leads to $2$ divides $A$
and therefore $N_S(\alpha-\xi_0)\geq (2^2+d)/18\geq 1$ when $d\geq 14$.
We will return to the particular cases of $d=10,13$ later.


Next, consider the case where $-d\equiv 1\pmod {4}$.
We assume, in addition, that $-d\not\equiv 1\pmod{8}$ so that $(1+d)/4$ is an odd integer.
We have
\begin{equation}\label{eqn:newnum}
  N(\alpha-\xi_0)=\frac{2^{2m}(A^2+AB+\left(\frac{1+d}{4}\right)B^2)}{3^2\cdot 2^{2n}}
  \,.
\end{equation}
We claim that $A^2+AB+((1+d)/4)B^2$ is not divisible by $2$;
indeed, if this were the case, then $A^2+AB+B^2\equiv 0\pmod{2}$
which leads to $A,B\equiv 0\pmod{2}$, a contradiction.
Therefore
\begin{equation}\label{eqn:latter}
  N_S(\alpha-\xi_0)\geq\frac{A^2+AB+\left(\frac{1+d}{4}\right)B^2}{3^2}
  =
  \frac{(2A+B)^2+dB^2}{4\cdot 3^2}
  \geq \frac{1+d}{36}
  \,,
\end{equation}
Hence $N_S(\alpha-\xi_0)\geq 1$ when $d\geq 35$.

Now we return to the case of $d=13$.  We choose $\xi_0=w/3$.
In a similar manner as before,
we write
$$
  N(\alpha-\xi_0)=\frac{2^{2m}(A^2+13B^2)}{2^{2n}6^2}
  \,,
$$
where $2^mA=3a$, $2^mB=3b-2^n$, $A,B$ not both even,
and $A^2+13B^2$ is divisible by at most one power of $2$.
Notice that $B\neq 0$ but that $A$ could equal $0$.
If $2$ does not divide $A^2+13B^2$, then $N_S(\alpha-\xi_0)\geq 13/9\geq 1$.
Hence we may assume $2$ divides $A^2+13B^2$.  In this case, it follows that $A$ and $B$ are both odd,
and since $A$ is a multiple of $3$, we have $N_S(\alpha-\xi_0)\geq (3^3+13)/18\geq 1$.

It turns out the case of $(d,p)=(10,2)$ will require a little more care.
We return to this case in the next section.
\end{proof}

It may be useful to remark that we have now completely dealt with the ``only if''
portion of Theorem~\ref{T:3}.  The ``if'' portion of the theorem follows from Theorem~\ref{T:1}
except for the five $(d,p)$ pairs which we will now consider.

\section{Exceptional cases}\label{S:except}
Here we deal with the cases where $(d,p)$ is one of the following pairs:
$$(10,2), (15,3), (15,5), (35,5), (35,7)
\,.$$

It turns out that our criterion from \S\ref{S:main}
just barely fails for the cases $d=10$ and $d=15$.
In particular, the countable union of intervals $\bigcup_{j,k}I_j^k$ fail to cover
$[0,1]$ by just finitely many points.  The case of $d=35$ is a little more difficult to deal with.

\begin{proposition}\label{P:1}
$\Q(\sqrt{-10})$ is $\{2\}$-norm-Euclidean.
\end{proposition}

\begin{proof}
For the first part of the proof, we will show that
\begin{equation}\label{E:claim}
\bigcup_{j,k}I_j^k \supseteq
[0,1]\setminus\left\{\frac{1}{3},\frac{2}{3}\right\}
\,.
\end{equation}

To establish this, we show that the condition in Corollary~\ref{C:main} holds for all $y\in[0,1]$ except when $y\in\{1/3,2/3\}$.
First, observe that $y\in\{1/3,2/3\}$ implies the fractional part $\{2^ny\}$ belongs to $\{1/3,2/3\}$ for all $n\geq 0$,
and hence the condition clearly fails for these two values of $y$.  Note that $\sqrt{3/40}\approx 0.274<1/3$.

Let $y\in[0,1]$ be arbitrary.  If there exists $n\geq 0$ such that $\{2^ny\}\in[0,1/4]\cup[3/4,1]$, then we are done,
since $1/4<\sqrt{3/40}$.  By way of contradiction, suppose that $\{2^n y\}\in[1/4,1/2]$ for all $n\geq 0$.
We have $y\in J_0:=[1/4,1/2]$ or $y\in J'_0:=[1/2, 3/4]$.  For sake of concreteness, suppose $y\in J_0$, but
the case of $y\in J'_0$ is treated in a similar manner.  Since $y\in [1/4,1/2]$, we have $2y\in[1/2,1]$ but by hypothesis it must be that $2y\in J_1:=[1/2,3/4]$.  Similarly, $4y\in J_2:=[5/4,3/2]$.  We inductively define a sequence of intervals $J_0,J_1,J_2,\dots$ in this manner, all of length $1/4$,
having the property that $2^ny\in J_n$.  Therefore $y\in\bigcap_n 2^{-n}J_n$, and since this intersection contains at most one element,
we get a contradiction, except for one possible value of $y$.  Similarly, if $y\in J'_0$, we can inductively define $J'_0,J'_1,J'_2,\dots$
and obtain a contradiction, except for one possible value of $y$.  Hence we have proven that the condition in Corollary~\ref{C:main}
holds with at most two possible exceptions.  But since we already know $1/3$ and $2/3$ are exceptions, this proves
(\ref{E:claim}).

It suffices to show that given $\xi=r/s+w/3\in\mathcal{F}$ with $s$ odd, there exists $\alpha\in\O_S$ such that $N_S(\xi-\alpha)<1$.
(If $\xi$ is of the form $r/s+2w/3$, then multiply by $2$ and translate back into $\mathcal{F}$ using an element of $\O$.)
Notice that $0\leq r/s\leq 1$ since $\xi\in\mathcal{F}$.  First set $\alpha=w/2$.  We find
$$
  \alpha-\xi_0=\frac{6r+sw}{6s}
  \,,\quad
  N(\alpha-\xi_0)  = \frac{18r^2+5s^2}{2\cdot 3^2 s^2}
$$
Therefore $N_S(\alpha-\xi_0)<2(r/s)^2+5/9<1$ when $r/s<\sqrt{2}/3\approx 0.47$.
If $\alpha=(2+w)/2$, then we find $N_S(\alpha-\xi_0)<2(1-r/s)^2+5/9<1$ when $r/s>1-\sqrt{2}/3\approx 0.53$.
If $\alpha=(2+w)/4$, then $N_S(\alpha-\xi_0)\leq 4(1/2-r/s)^2+5/9<1$ when $1/6\leq r/s\leq 5/6$.
This completes the proof.
\end{proof}

\begin{proposition}
$\Q(\sqrt{-15})$ is $\{3\}$-norm-Euclidean and $\{5\}$-norm-Euclidean.
\end{proposition}

\begin{proof}
Suppose $S=\{p\}$ with $p=3$ or $p=5$.
We proceed in a similar manner as in Proposition~\ref{P:1}.
First we claim that $\bigcup_{j,k}I_j^k \supseteq[0,1]\setminus\left\{1/2\right\}$.
For sake of concreteness, we prove this when $p=3$.
Let $y\in[0,1]$ be arbitrary.  Suppose $\{3^n y\}\in[1/3,2/3]$ for all $n\geq 0$.
(If this condition does not hold, then we are done since $1/3<\sqrt{3/15}$.)
This means that $y\in J_0:=[1/3,2/3]$, $3y\in J_1:=[4/3,5/3]$, $9y\in J_2:=[13/3,14/3]$, and so on.
Inductively, we find
$$
3^n y\in J_n:=\left[\frac{3^n}{2}-\frac{1}{6}, \frac{3^n}{2}+\frac{1}{6}\right]\,.
$$
Therefore
$y\in \bigcap_n 3^{-n}J_n =\{1/2\}$.  This proves the claim when $p=3$.
When $p=5$, the proof is similar.  Noting that $2/5<\sqrt{3/15}$,
we can assume that $\{5^n y\}\in[2/5,3/5]$ for all $n\geq 0$.  This leads to
$J_n:=[5^n/2-1/10, 5^n/2+1/10]$ and the result follows as before.

Let $\xi=r/s+w/2\in\mathcal{F}$ with $(p,s)=1$.
We may assume $0\leq r/s<1$.
If $\alpha=w$, then $N_S(\xi-\alpha)\leq(r/s-1/4)^2+15/16<1$ when $0<r/s<1/2$.
Similarly, if $\alpha=1$, then $N_S(\xi-\alpha)\leq (r/s-3/4)^2+15/16<1$ when $1/2<r/s<1$.
This leaves only the points $\xi=w/2$ and $\xi=(1+w)/2$.
When $\xi=w/2$, we compute $N(\xi)=15/16$ so that $N_S(\xi)<1$ in either case.
Consider $\xi=(1+w)/2$.
We have $N(\xi)=3/2$ so that $N_S(\xi)=1/2<1$ in the case $p=3$.
For the case $p=5$, we compute $N_S(\xi-2)=5/2$ so that $N_S(\xi-2)=1/2<1$.
This completes the proof.
\end{proof}

It appears that the strategy employed for the previous two propositions
does not work when $d=35$.
Indeed, although it turns out that $K$ is $S$-norm-Euclidean
in the remaining exceptional cases,
Lemma~\ref{lem:2} may not be sufficient to prove this.
We require a new idea.  The following lemma essentially says
that in some cases, we can increase the radii of our circles.

\begin{lemma}\label{lem:boost}
Notation as in Lemma~\ref{lem:2}.
Assume $p\in S$ and $2\not\in S$.
Assume $-d\equiv 1\pmod{4}$ and $p\mid d$.
If the following holds,
\begin{equation}\label{eqn:cong}
c\equiv0\pmod{p}\,,\;\;
b\not\equiv 0\pmod{p}\,,\;\;
2a+b\equiv 0\pmod{p}
\end{equation}
then
$$
  N_S(\xi-\alpha)\leq \frac{c^2}{p}N(\xi-\alpha)
  \,.
$$
In particular, in Lemma~\ref{lem:3}, it suffices to find a value of $\alpha$
satisfying (\ref{eqn:cong}) such that
$$|\xi-\alpha|<\frac{\sqrt{p}}{c}\,.$$
\end{lemma}

\begin{proof}
We have
$\xi-\alpha=(A+Bw)/(cz)$
where $A=cx-az$ and $B=cy-bz$
so that
$$
  N(\xi-\alpha)=\frac{(2A+B)^2+dB^2}{4c^2z^2}
  \,.
$$
As in the proof of Lemma~\ref{lem:2} we have
$$
  N_S(\xi-\alpha)\leq \frac{(2A+B)^2+dB^2}{4z^2}=c^2 N(\xi-\alpha)
$$
If in addition, the conditions (\ref{eqn:cong}) hold, then
$2A+B\equiv 0\pmod{p}$ which implies the numerator is divisible by $p$,
and therefore
$N_S(\xi-\alpha)\leq p^{-1}c^2 N(\xi-\alpha)$.
\end{proof}

\begin{proposition}
$\Q(\sqrt{-35})$ is $\{5\}$-norm-Euclidean and $\{7\}$-norm-Euclidean.
\end{proposition}

\begin{proof}
Set $S=\{p\}$.
We first consider $p=5$.
Making use of Lemmas~\ref{lem:3} and~\ref{lem:boost} we obtain a list of $14$
circles of various radii that suffice to cover $\mathcal{F}$.  See Table~\ref{table:A}
for the centers of the chosen circles (as elements $\alpha\in\O_S$) and the radius afforded 
by our lemmas.
\begin{table}
\begin{tabular}{c|c}
$r$ & $\alpha$\\
\hline\\[0ex]
$1$ & $0$, $1$, $w$, $1+w$\\[2ex]
\hline\\
$1/5$ & $\displaystyle\frac{1+2w}{5}$, $\displaystyle\frac{2+2w}{5}$, $\displaystyle\frac{3+3w}{5}$, $\displaystyle\frac{4+3w}{5}$\\[2ex]
\hline\\
$\sqrt{5}/5$ & $\displaystyle\frac{2+w}{5}$, $\displaystyle\frac{-1+2w}{5}$, $\displaystyle\frac{6+3w}{5}$,
$\displaystyle\frac{3+4w}{5}$, $\displaystyle\frac{4+2w}{5}$, $\displaystyle\frac{1+3w}{5}$\\
\end{tabular}
\caption{Circles $\{|\xi-\alpha|<r\}$ that cover $\mathcal{F}$ when $p=5$}
\label{table:A}
\end{table}
In principle, one could verify by hand that the $14$ provided circles cover $\mathcal{F}$.
We believe the picture in Figure~\ref{fig3} that displays the covering is fairly convincing.
If one is not convinced by the picture, one could partition the fundamental domain into $5^6$
parallelograms of equal area in the obvious way, and use a computer to check
that each parallelogram lies entirely inside one of the given circles.
When $p=7$, the argument is similar.  We give a list of $20$ circles (see Table~\ref{table:B})
that carry out the covering (also depicted in Figure~\ref{fig3}).
\begin{table}
\begin{tabular}{c|c}
$r$ & $\alpha$\\
\hline\\
$1$ & $0$, $1$, $w$, $1+w$\\[2ex]
\hline\\
$1/7$ & $\displaystyle\frac{3+2w}{7}$, $\displaystyle\frac{5+3w}{7}$, $\displaystyle\frac{6+3w}{7}$, $\displaystyle\frac{7+3w}{7}$, $\displaystyle\frac{4w}{7}$, $\displaystyle\frac{1+4w}{7}$, $\displaystyle\frac{2+4w}{7}$, $\displaystyle\frac{5+5w}{7}$\\[2ex]
\hline\\
$\sqrt{7}/7$ & $\displaystyle\frac{3+w}{7}$, $\displaystyle\frac{-1+2w}{7}$, $\displaystyle\frac{6+2w}{7}$,
$\displaystyle\frac{2+3w}{7}$, $\displaystyle\frac{5+4w}{7}$, $\displaystyle\frac{1+5w}{7}$,
$\displaystyle\frac{8+5w}{7}$, $\displaystyle\frac{4+6w}{7}$\\
\end{tabular}
\caption{Circles $\{|\xi-\alpha|<r\}$ that cover $\mathcal{F}$ when $p=7$}
\label{table:B}
\end{table}
\begin{figure}[h!]
  \centering
  \begin{subfigure}[b]{0.4\linewidth}
    \includegraphics[width=\linewidth]{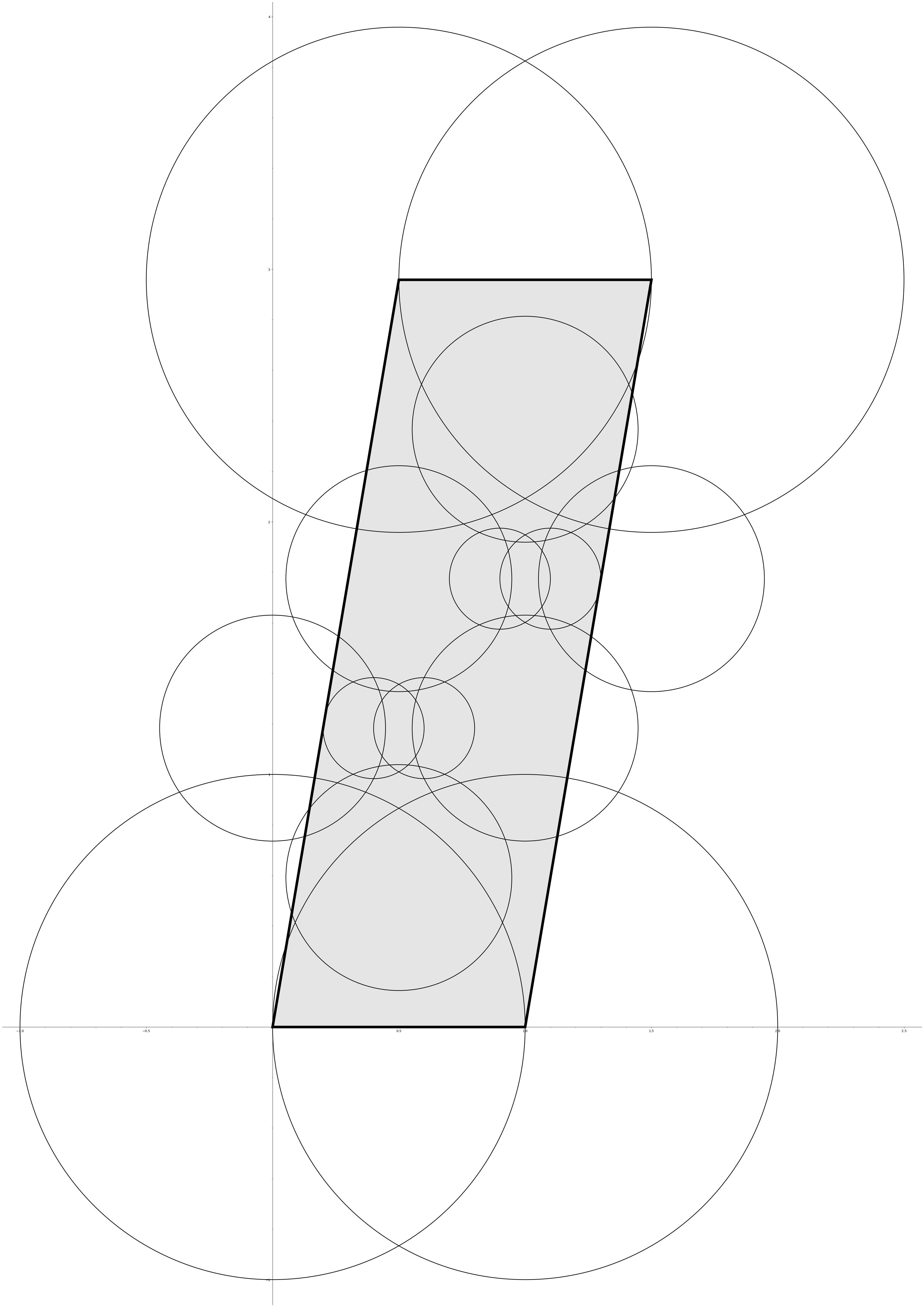}
    \caption{Covering of $\mathcal{F}$ when $p=5$}
  \end{subfigure}
  \begin{subfigure}[b]{0.4\linewidth}
    \includegraphics[width=\linewidth]{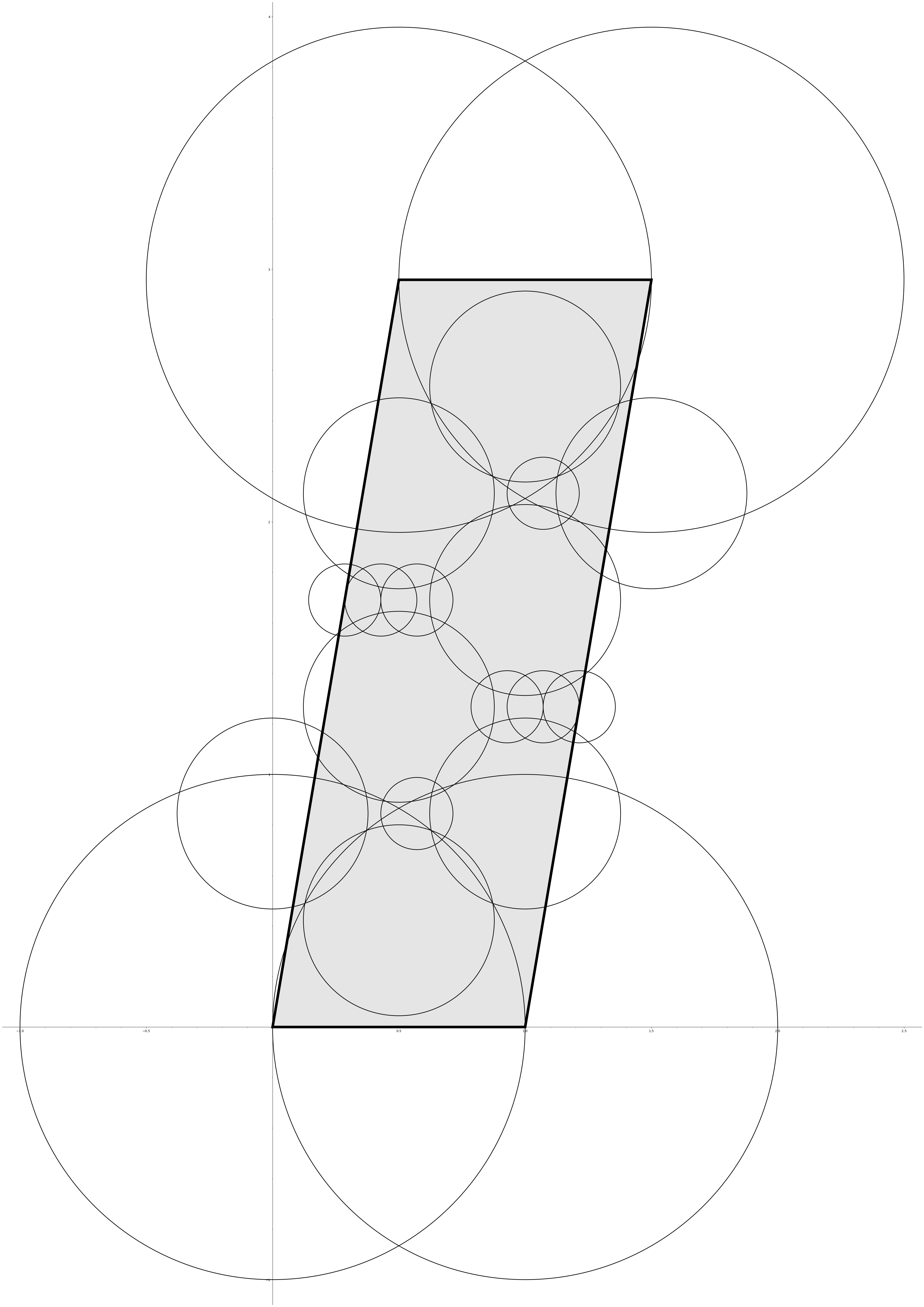}
    \caption{Covering of $\mathcal{F}$ when $p=7$}
  \end{subfigure}
  \caption{$K=\Q(\sqrt{-35})$}
      \label{fig3}
\end{figure}
\end{proof}


%
%
%
%


\section{Appendix}\label{S:appendix}

We finish with some discussion that requires a modicum of algebraic number theory.
Let $K$ be a number field with ring of integers $\O$, and let $N$ denote the absolute value of the norm map.
In the quadratic setting, it is known that there are only finitely many norm-Euclidean fields.
As we have seen, when $K$ is complex quadratic, this result is classical.
When $K$ is real quadratic,
the classification was completed by Chatland and Davenport in 1950 (see~\cite{CD}).
It turns out that $K=\Q(\sqrt{d})$ with $d>0$ squarefree is norm-Euclidean if and only if
$d=2, 3, 5, 6, 7, 11, 13, 17, 19, 21, 29, 33, 37, 41, 57, 73$.
Barnes and Swinnerton-Dyer corrected an error from earlier work (see~\cite{BS});
for a time, it was thought that the field $\Q(\sqrt{97})$ was norm-Euclidean.
See~\cite{ELS} for a self-contained exposition of this result.
See~\cite{LENSTRA1,LENSTRA23} for an introduction to the topic of Euclidean number fields,
or~\cite{LEMMERMEYER} for a survey of the field.

Recall that every number field $K$ has $n=[K:\Q]$ embeddings $\sigma:K\to\C$.
We have $n=r_1+2r_2$ where $r_1$ is the number of real embeddings and $2r_2$ is the number
of conjugate pairs of complex embeddings.  We say that unit group $\O^\times$ has rank $r$,
and write $\rk(\O^\times)=r$, if $\O^\times$ modulo its torsion subgroup is isomorphic
to $\Z^r$.  Dirichlet's Unit Theorem states that $\rk(\O^\times)=r_1+r_2-1$.
Davenport generalized the result of~\cite{CD} to show that there are only finitely many norm-Euclidean fields
with $\rk(\O^\times)=1$.
Note that $\rk(\O^\times)=1$ iff $(r_1,r_2)=(2,0),(1,1),(0,2)$.  Very little is known 
about norm-Euclidean fields when $\rk(\O^\times)>1$.

We will now move to the $S$-integral setting,
but first we introduce a little more terminology and notation.
(One possible reference for this material is~\cite{lang}.)
Recall that associated to every prime ideal $\mathfrak{p}$ is a valuation
$v:K\setminus\{0\}\to\Z$, where $v_\mathfrak{p}(\alpha)$ is the power of $\mathfrak{p}$
in the (unique) factorization of the fractional ideal $(\alpha)$ into prime ideals of $\O$.
We refer to equivalence classes of embeddings up to complex conjugation as infinite primes.
The set of all infinite primes is denoted by $S_\infty$; notice that $\#S_\infty=r_1+r_2$
and hence $\rk(\O^\times)=\#S_\infty-1$.
By a prime of $K$ we will mean either a finite prime (which is a prime ideal $\mathfrak{p}$)
or an infinite prime (which is an embedding $\sigma$).  
We can associate to every prime $v$ of $K$ (finite or infinite) an absolute value $|\cdot|_v:K\to\R_{\geq 0}$.
If the prime $v$ is finite (associated to an ideal $\mathfrak{p}$) then $|\alpha|_v=(N\mathfrak{p})^{-v_\mathfrak{p}(\alpha)}$.
If the prime $v$ is infinite (associated to an embedding $\sigma$) then $|\alpha|_v=|\sigma(\alpha)|$
or $|\alpha|_v=|\sigma(\alpha)|^2$, depending upon whether $\sigma$ is real or complex.
This normalization is chosen so that $N(\alpha)=\prod_{v\in S_\infty}|\alpha|_v$.

We give the usual definition of the ring of $S$-integers.
Let $S$ be a finite set of primes of $K$ containing $S_\infty$.
We define the $S$-integers of $K$ as
$$
\O_S=\{\xi\in K\mid v(\xi)\geq 0\text{ for all }v\notin S\}
\,.
$$
It should be noted that passing from $\mathcal{O}$ to $\mathcal{O}_{S}$ has the effect of inverting all the prime ideals in $S$.
The $S$-norm of a number $\xi\in K$ is defined as $N_S(\xi)=\prod_{v\in S}|\xi|_v$.

A theorem of O'Meara (see~\cite{OMEARA}) says that that for any number field $K$ there exists a finite set of primes $S$ such that $K$ is $S$-norm-Euclidean.  It would be desirable to have a quantitative form of this theorem that holds for any number field.
In principle, the techniques developed by Lenstra in~\cite{LENSTRA} would allow one to prove a quantitative version of
this theorem (see the last paragraph of Section 1 of~\cite{LENSTRA}).
However, to our knowledge this has not been carried out anywhere in the literature.
On the other hand, O'Meara's Theorem is also proved in~\cite{MARKANDA};
in principle this approach could be made quantitative as well, but the author only explicitly states such a result
in the complex quadratic case.

The determination of all $S$-norm-Euclidean number fields (and function fields) with $\#S=2$ was completed by van der Linden (see~\cite{VDL}).  (The case of $\#S=1$ is an exercise.)  The analogue of Dirichlet's Unit Theorem in the $S$-integral setting
says that
$\rk(\O_S^\times)=\#S-1$ and consequently, $\#S=2$ is precisely the
situation where $\O_S^\times$ has rank one.  As indicated before, not much is known when $\rk(\O_S^\times)>1$.

We now explain how the setting in which we have been working throughout the paper fits into this picture.
Let $S$ be a finite set of rational primes as in~\S\ref{S:intro}.
Let $S_K$ denote the set all prime ideals of $K$ lying above primes in $S$
together with the set of all infinite primes $S_\infty$.
In this setting one could write $\mathcal{O}_{S}$ instead of the more cumbersome but accurate $\mathcal{O}_{S_K}$,
and $N_S$ instead of $N_{S_K}$.  It turns out that $\O_S$ is then the same ring we introduced in~\S\ref{S:intro},
and $N_S$ is the same function.  We leave these claims as exercises to the reader.
It can now be seen that our proof of Theorem~\ref{T:2} is an elementary proof of a quantitative
version of O'Meara's Theorem in the complex quadratic case.

Let $K$ be a complex quadratic field with notation as in~\S\ref{S:intro}.
Let $p$ be a rational prime.  Recall that one calls $p$ split, inert, or ramified in $K$ depending upon whether the
factorization of $(p)$ into prime ideals looks like $(p)=\mathfrak{p}\overline{\mathfrak{p}}$, $(p)=\mathfrak{p}$, $(p)=\mathfrak{p}^2$,
respectively.  When $p$ is odd this can be detected by observing whether the Legendre symbol $(-d/p)$
is equal to $1$, $-1$ or $0$; when $p=2$, this can de detected by whether $-d\pmod{8}$
belongs to $\{1\}$, $\{5\}$, or $\{2,3,6,7\}$.  As in Theorem~\ref{T:3}, consider $S=\{p\}$.
Then $\#S_K=2$ when $p$ is inert or ramified, and $\#S_K=3$ if $p$ splits.
In other words, the conditions in Theorem~\ref{T:3} were secretly insisting that $\rk(\O_S^\times)=1$.
It follows that our Theorem~\ref{T:3} is subsumed by van der Linden's classification.
We still believe it is worthwhile to put down as our proof is elementary in nature.
On the other hand, Theorem~\ref{T:1} contains a number of examples in which $\rk(\O_S^\times)>1$.

We now give the connection with the class group of $K$, denoted by $\Cl(K)$.
We write $I(K)$ to denote the multiplicative group of nonzero fractional ideals
and $P(K)$ denote the subgroup of principal fractional ideals, so that
$\Cl(K)=I(K)/P(K)$.
It is clear that $\Cl(K)$ is trivial if and only if $\O$ is a PID; in this setting,
this is equivalent to $\O$ being a UFD.
Suppose  $S$ be a set of primes in $K$ containing $S_\infty$.
We can analogously define the $S$-class group $\Cl_S(K)=I_S(K)/P_S(K)$.
(We note that $\O_S$, like $\O$, is a Dedekind domain.)  It turns out that $\Cl_S(K)$ is isomorphic
to $\Cl(K)$ modulo the subgroup generated by the finite primes in $S$.
In particular, if $K$ is $S$-norm-Euclidean, then $\Cl_S(K)$ is trivial, and hence
the primes (or even just the non-principal primes) in $S$  generate $\Cl(K)$.
This explains the comment immediately following Theorem~\ref{T:2};
indeed, in the complex quadratic setting, the Minkowski constant is $2/\pi>1/\sqrt{3}$.

Finally, consider the situation of Theorem~\ref{T:3} and the remark that follows.
Let $K$ be a complex quadratic field and $S=\{p\}$.  If $p$ is inert,
then $\Cl(K)\simeq\Cl_S(K)$ and hence $K$ being $S$-norm-Euclidean implies
$\Cl(K)$ is trivial.  A famous thoerem of Baker--Heegner--Stark
states that this only happens for $d=1,2,3,7,11,19,43,67,163$.
If $p$ is ramified then $(p)=\mathfrak{p}^2$ and $\Cl_S(K)$ is isomorphic
to $\Cl(K)/\langle\mathfrak{p}\rangle$ where $\mathfrak{p}^2$ is trivial in $\Cl(K)$;
consequently, $K$ being $S$-norm-Euclidean
implies that $\Cl(K)$ has order $1$ or $2$.
Similarly, complex quadratic fields with class number $2$ have been classified,
so the list of possibilities is finite.
(Keep in mind that the solutions to these class number problems
are deep results and that our proof of Theorem~\ref{T:3} did not make use of these results
in any way.)
On the other hand, if $p$ is split then $(p)=\mathfrak{p}\overline{\mathfrak{p}}$,
and since $\mathfrak{p}$, $\overline{\mathfrak{p}}$ are inverses in $\Cl(K)$,
it follows again that $\Cl_S(K)\simeq\Cl(K)/\langle\mathfrak{p}\rangle$;
consequently, $K$ being $S$-norm-Euclidean
implies the class group is cyclic generated by $\mathfrak{p}$.
This event should happen infinitely often, but currently this is unproven.
This explains the discussion following Theorem~\ref{T:3}.

%

\section*{Acknowledgment}
This research was completed as part of the Research Experience for Undergraduates and Teachers program at California State University, Chico funded by the National Science Foundation (DMS-1559788).

\bibliographystyle{plain}
\bibliography{mybib}

\begin{thebibliography}{10}

\bibitem{BS}
E.~S. Barnes and H.~P.~F. Swinnerton-Dyer.
\newblock The inhomogeneous minima of binary quadratic forms. {I}.
\newblock {\em Acta Math.}, 87:259--323, 1952.

\bibitem{CD}
H.~Chatland and H.~Davenport.
\newblock Euclid's algorithm in real quadratic fields.
\newblock {\em Canad. J. Math.}, 2:289--296, 1950.

\bibitem{CL}
H.~Cohen and H.~W. Lenstra, Jr.
\newblock Heuristics on class groups.
\newblock In {\em Number theory ({N}ew {Y}ork, 1982)}, volume 1052 of {\em
  Lecture Notes in Math.}, pages 26--36. Springer, Berlin, 1984.

\bibitem{DD}
P.~G.~Lejeune Dirichlet.
\newblock {\em Vorlesungen \"{u}ber {Z}ahlentheorie}.
\newblock Herausgegeben und mit Zus\"{a}tzen versehen von R. Dedekind. Vierte,
  umgearbeitete und vermehrte Auflage. Chelsea Publishing Co., New York, 1968.

\bibitem{ELS}
R.~B. Eggleton, C.~B. Lacampagne, and J.~L. Selfridge.
\newblock Euclidean quadratic fields.
\newblock {\em Amer. Math. Monthly}, 99(9):829--837, 1992.

\bibitem{lang}
Serge Lang.
\newblock {\em Algebraic Number Theory}.
\newblock Applied Mathematical Sciences. Springer, 1994.

\bibitem{LEMMERMEYER}
Franz Lemmermeyer.
\newblock The {E}uclidean algorithm in algebraic number fields.
\newblock {\em Exposition. Math.}, 13(5):385--416, 1995.

\bibitem{LENSTRA}
H.~W. Lenstra, Jr.
\newblock Euclidean number fields of large degree.
\newblock {\em Invent. Math.}, 38(3):237--254, 1976/77.

\bibitem{LENSTRA1}
Hendrik~W. Lenstra, Jr.
\newblock Euclidean number fields. {I}.
\newblock {\em Math. Intelligencer}, 2(1):6--15, 1979/80.

\bibitem{LENSTRA23}
Hendrik~W. Lenstra, Jr.
\newblock Euclidean number fields. {II}, {III}.
\newblock {\em Math. Intelligencer}, 2(2):73--77, 99--103, 1979/80.
\newblock Translated from the Dutch by A. J. Van der Poorten.

\bibitem{MARKANDA}
Raj Markanda.
\newblock Euclidean rings of algebraic numbers and functions.
\newblock {\em J. Algebra}, 37(3):425--446, 1975.

\bibitem{OMEARA}
O.~T. O'Meara.
\newblock On the finite generation of linear groups over {H}asse domains.
\newblock {\em J. Reine Angew. Math.}, 217:79--108, 1965.

\bibitem{STARK}
H.~M. Stark.
\newblock The {G}auss class-number problems.
\newblock In {\em Analytic number theory}, volume~7 of {\em Clay Math. Proc.},
  pages 247--256. Amer. Math. Soc., Providence, RI, 2007.

\bibitem{VDL}
F.~J. van~der Linden.
\newblock {\em Euclidean rings with two infinite primes}, volume~15 of {\em CWI
  Tract}.
\newblock Stichting Mathematisch Centrum, Centrum voor Wiskunde en Informatica,
  Amsterdam, 1985.

\end{thebibliography}

\end{document}